\documentclass[english]{amsart}

\usepackage{amsmath}
\usepackage{pifont}
\usepackage{amssymb}

\usepackage{amscd}
\usepackage[all,cmtip,line]{xy}

\newcommand{\Vbar}{\overline{V}}

\newcommand{\onto}{\twoheadrightarrow}

\newlength{\ownl}

\newcommand{\End}{{\operatorname{End}\,}}

\newcommand{\Frob}{{\operatorname{Frob}}}

\newcommand{\Gal}{{\operatorname{Gal}\,}}

\newcommand{\ind}{{\operatorname{ind}}}

\newcommand{\Symm}{{\operatorname{Symm}}}

\newcommand{\GL}{\operatorname{GL}}

\newcommand{\F}{{\mathbb{F}}}

\newcommand{\Q}{{\mathbb{Q}}}

\newcommand{\Z}{{\mathbb{Z}}}

\newcommand{\Thetabar       }{\overline{\Theta}}

\def\RCS$#1: #2 ${\expandafter\def\csname RCS#1\endcsname{#2}}
\RCS$Revision: 327 $
\RCS$Date: 2010-11-15 13:46:16 -0600 (Mon, 15 Nov 2010) $

\newcommand{\Qp}{\Q_p}
\newcommand{\GQp}{G_{\Qp}}
\newcommand{\Zp}{\Z_p}
 
\newcommand{\Qpbar}{\overline{\Q}_p}

\newcommand{\Fpbar}{\overline{\F}_p}

\newcommand{\Fp}{{\F_p}}

\newcommand{\Zpbar}{\overline{\Z}_p}

\def\smallmat#1#2#3#4{\bigl(\begin{smallmatrix}{#1}&{#2}\\{#3}&{#4}\end{smallmatrix}\bigr)}

\usepackage{hyperref} 

\usepackage{amsthm}
 
\newtheorem{ithm}{Theorem}
\newtheorem{thm}{Theorem}[section]

\newtheorem{corollary}[thm]{Corollary}
\newtheorem{cor}[thm]{Corollary}
 \newtheorem{lemma}[thm]{Lemma}
\newtheorem{lem}[thm]{Lemma}

 \theoremstyle{definition}
 \theoremstyle{definition}
\newtheorem{defn}[thm]{Definition} \theoremstyle{remark}

\theoremstyle{definition}

\begin{document}
\title[Explicit reduction]{Explicit reduction modulo $p$ of certain $2$-dimensional
  crystalline representations, II}

\author{Kevin Buzzard} \email{buzzard@imperial.ac.uk} \address{Department of
  Mathematics, Imperial College London}  \author{Toby Gee} \email{toby.gee@imperial.ac.uk} \address{Department of
  Mathematics, Imperial College London}   
\begin{abstract}We complete the calculations begun in
  \cite{MR2511912}, using the $p$-adic local Langlands correspondence
  for $\GL_2(\Qp)$ to give a complete description of the reduction
  modulo $p$ of the $2$-dimensional crystalline representations of
  $\GQp$ of slope less than $1$, when $p>2$.

\end{abstract}
\maketitle

\section{Introduction}\label{sec:intro}
This paper is a sequel to \cite{MR2511912}, and we refer the reader to
the introduction to that paper for a detailed discussion of (and the
motivation for) the problem solved in this paper. Another good
reference is \S5.2 of~\cite{bergerSemBourb}. 
Let $p$ be a prime,
choose an algebraic closure $\Qpbar$ of $\Qp$,
let $\Zpbar$ be the integers in $\Qpbar$ and let $\Fpbar$ be the
residue field of $\Zpbar$. We let $v$ be the $p$-adic valuation on
$\Qpbar^\times$, normalised so that $v(p)=1$. We set $v(0)=+\infty$.
We decree that the cyclotomic
character has Hodge--Tate weight $+1$.
We recall that given a positive
integer $k\ge 2$ and an element $a\in\Qpbar$ with $v(a)>0$ there is a
uniquely determined two-dimensional crystalline representation
$V_{k,a}$ of $\Gal(\Qpbar/\Qp)$ with Hodge--Tate weights $0$ and
$k-1$, determinant the cyclotomic character to the power of
$k-1$, and with the characteristic polynomial of crystalline
Frobenius on the contravariant Dieudonne module being $X^2-aX+p^{k-1}$
(see for example \S3.1 of~\cite{breuil2} for a detailed construction
of this representation). 
Let $\Vbar_{k,a}$ denote the
semisimplification of the reduction of $V_{k,a}$ modulo the maximal
ideal of $\Zpbar$. Let $\omega$ denote the mod~$p$ cyclotomic
character, and if $p+1\nmid n$ let $\ind(\omega_2^n)$ denote the unique
irreducible
2-dimensional representation of $\GQp$ with determinant $\omega^n$
and with restriction to inertia equal to $\omega_2^n\oplus\omega_2^{pn}$,
with $\omega_2$ the ``niveau~2'' character of inertia (see for
example \S1.1 of~\cite{bergerSemBourb}).

Our main result is the following, which is an
immediate consequence of Theorem 1.6 of~\cite{MR2511912} (the case
$k\not\equiv3$~mod~$p-1$), Theorem~3.2.1 of~\cite{MR2642408} (the
cases $k=3$ and $k=p+2$), and Corollary
\ref{cor: explicit description of Vbar in the 3 cases} below.
Recall $k\geq2$; let $[k-2]$
denote denote the integer in the set $\{0,1,\ldots,p-2\}$ congruent to $k-2$
mod $p-1$, and set $t=[k-2]+1$, so $1\leq t\leq p-1$.

\begin{ithm}Assume that $p>2$ and that $0<v(a)<1$. Then
  $\Vbar_{k,a}\cong\ind(\omega_2^t)$ is irreducible, unless $k>3$,
$k\equiv 3
  \pmod{p-1}$, and $v(k-3)+1+v(a)\le v(a^2-(k-2)p)$, in which case
  $\omega^{-1}\otimes\Vbar_{k,a}$ is unramified, and the trace of
a geometric Frobenius $\Frob_p$
 on $\omega^{-1}\otimes\Vbar_{k,a}$ is $\overline{\tau}$, where
  $\tau=\frac{(k-2)p-a^2}{ap(k-3)}$.
\end{ithm}
Note that when $k\equiv 3 \pmod{p-1}$ and $v(k-3)+1+v(a)\le v(a^2-(k-2)p)$,
the $\tau$ in the theorem is in $\Zpbar$, and its reduction is also
the trace of an arithmetic Frobenius, because $\omega^{-1}\otimes\Vbar_{k,a}$
has trivial determinant in this case.

For a fixed~$k$ one can look at the behaviour of the representation
$V_{k,a}$ as $a$ varies through the annulus $0<|a|<1$, and we can give
a more prosaic description of what our theorem says: it
says that often $V_{k,a}=\ind(\omega_2)$ is constant on this
annulus, the only exception being when $k>3$, $k\equiv3$~mod~$p-1$
and $k\not\equiv2$~mod~$p$, in which case the representation
is $\ind(\omega_2)$ everywhere other than two small closed
discs with centre $\pm\sqrt{(k-2)p}$ and radius $p^{-1-v(k-3)}$.
Note that both these small discs are contained in the annulus $v(a)=1/2$,
and that as $k>3$ tends to 3 $p$-adically the radius of the discs tends
to zero. Note also that the limit of $\pm\sqrt{(k-2)p}$
as $k>3$ tends to~3 $p$-adically is $\pm\sqrt{p}$, however $a=\pm\sqrt{p}$
is not in any of the discs; furthermore the intersection of all these
discs as $k$ varies is empty, and in particular our result does not
contradict the local constancy results of~\cite{BergerLocCst},
contrary to one's initial reaction.

This theorem was proved in the case $k\not\equiv 3\pmod{p-1}$ in
\cite{MR2511912} using the $p$-adic local Langlands correspondence for
$\GL_2(\Qp)$. In the present paper we build on the results and methods
of \cite{MR2511912} to handle the case $k\equiv 3\pmod{p-1}$; as one
might expect from the statement of the theorem, the
necessary calculations are more complicated in this case, because
we have to control what is going on modulo an arbitrarily large
power of~$p$ in the auxiliary calculations.

We would like to thank Christophe Breuil for sharing with us the
details of his unpublished calculations for $k=2p+1$, which were the
starting point for this article. We would also like to thank
Mathieu Vienney for pointing out a howler of a typo
in the statement of the main theorem in an earlier version
of this paper, and John Enns for pointing out another typo later in
the paper. We would also like to thank the anonymous referee for a
careful reading, and several helpful corrections and improvements.

\subsection{Notation} Throughout the paper, $p$ denotes an odd
prime, and $r$ and~$n$ are integers.
If $\lambda\in\Fp$, we write $[\lambda]\in\Zp$ for its
Teichmueller lift.

\section{Combinatorial Lemmas}\label{sec:combinatorial lemmas}In this
section we prove some elementary lemmas about congruences of binomial
coefficients, that we will make repeated use of in the rest of the
paper.

\begin{lemma}\label{binomial}
  Assume that $p>2$ and that $r\in\Z_{\geq2}$, and write
  $t=v(r-1)\geq0$. Then \begin{enumerate}
  \item for all integers $n\geq2$ we have
$v\left(\binom{r}{n}\right)+n\geq t+2$, and
\item for all integers $n\geq1$, $v\left(\binom{r-1}{n}\right)+n\geq t+1$.
  \end{enumerate}

\end{lemma}
\begin{proof}
  \begin{enumerate}
  \item The left hand side is $v(r(r-1)(r-2)\cdots(r-(n-1)))-v(n!)+n$
which is at least $v(r-1)-v(n!)+n=t+n-v(n!)$. If $n=2$ then the
result holds as $p>2$.
If $n\geq3$ then we need to check $n-v(n!)\geq2$ but this
is clear because $v(n!)\leq n/(p-1)$ and hence $n-v(n!)\geq
n(p-2)/(p-1)\ge 3(p-2)/(p-1)$,
and it is enough to prove that $\lceil 3(p-2)/(p-1)\rceil\geq2$, which
is true (by an explicit check for $p=3$ and true even without the
$\lceil\cdot\rceil$ for $p\geq5$).
\item The left hand side is $v((r-1)(r-2)\cdots(r-n))-v(n!)+n\geq
v(r-1)-v(n!)+n=t+n-v(n!)$. Again, $v(n!)\leq n/(p-1)$, and hence
the left hand side is at least $t+n(p-2)/(p-1)$. Now
the result is true for $n=1$ so it suffices to prove
that $\lceil 2(p-2)/(p-1)\rceil\geq 1$, which follows as $p>2$.
  \end{enumerate}\end{proof}

\begin{lemma}\label{lem: r at least t+3}Assume
  that $p>2$ and that $r>1$, and write $t=v(r-1)\geq0$. Assume that
  $r\equiv 1\pmod{p-1}$. Then $r\geq t+3$.
\end{lemma}
\begin{proof} $r-1\geq (p-1)p^t\geq 2\cdot 3^t\geq t+2$ by easy induction.
\end{proof}

\begin{lemma}\label{first congruence mod p^t+2} Assume that $p>2$ and
  $r>1$, and that $r\equiv 1\pmod{p-1}$. Write $t=v(r-1)\geq0$. If
  $\mu\in\Fp$, then $(-[\mu] x+py)^r-x^{r-1}(-[\mu] x+py)$
  is congruent modulo $p^{t+2}\Zpbar[x,y]$ to
$$-px^{r-1}y$$
if $\mu=0$, and to
$$(r-1)px^{r-1}y$$
if $\mu\not=0.$
\end{lemma}
\begin{proof} If $\mu=0$ then we just need to check that $r\geq
  t+2$, which follows from Lemma \ref{lem: r at least t+3}.

  If however $\mu\not=0$ then we expand via the binomial theorem
  and use part Lemma \ref{binomial}(1) (and the fact that
  $(p-1)|(r-1)$, so $[\mu]^{r-1}=1$) to get that modulo $p^{t+2}$ we
  have
\begin{align*}
&\phantom{\equiv}(-[\mu] x+py)^r-x^{r-1}(-[\mu] x+py)\\
&\equiv-[\mu] x^r+rpx^{r-1}y+[\mu] x^r-x^{r-1}py\\
&\equiv(r-1)px^{r-1}y,
\end{align*}as required.
\end{proof}

\begin{lemma}\label{lem: sums of powers of 1+lambda} If $p>2$ and $r>1$ with $r\equiv1\pmod{p-1}$, and
if $t=v(r-1)$, then
\begin{enumerate}
\item  $\sum_{\mu\in\Fp}(1+[\mu])^r\equiv rp\pmod{p^{t+2}}$, and
\item  $\sum_{\mu\in\Fp}(1+[\mu])^{r-1}\equiv p-1\pmod{p^{t+1}}$.
\end{enumerate}

\end{lemma}
\begin{proof}

  (1) We rewrite $(1+[\mu])^r$ as
  $([1+\mu]+(1+[\mu]-[1+\mu]))^r$ and expand using the
  binomial theorem. Since $(1+[\mu]-[1+\mu])$ is divisible by
  $p$, by Lemma~\ref{binomial}(1) we only need to look at the first two
  terms in the binomial expansion to compute it modulo $p^{t+2}$, and
  we see that the sum is congruent modulo $p^{t+2}$ to
$$\sum_{\mu\in\Fp}\left([1+\mu]^r+r[1+\mu]^{r-1}(1+[\mu]-[1+\mu])\right).$$
Since $r\equiv 1\pmod{p-1}$, we have
$\sum_{\mu\in\Fp}[1+\mu]^r=\sum_{\mu\in\Fp}[\mu]=0$,
and since $[1+\mu]^{r-1}=1$ unless $\mu=-1$, and if
$\mu=-1$ then $1+[\mu]-[1+\mu]=0$, the sum is congruent
modulo $p^{t+2}$ to
\begin{align*}
\sum_\mu r[1+\mu]^{r-1}(1+[\mu]-[1+\mu])
&=r\sum_{\mu\not=-1}(1+[\mu]-[1+\mu])\\
&=r\sum_\mu(1+[\mu]-[1+\mu])\\
&=r\sum_\mu 1=rp
\end{align*}
and we are done.

(2) We do the same trick using Lemma~\ref{binomial}(2), which implies
that we only have to look at the first term of the binomial
expansion. Modulo $p^{t+1}$ we have
\begin{align*}
\sum_\mu(1+[\mu])^{r-1}
&=\sum_\mu([1+\mu]+(1+[\mu]-[1+\mu]))^{r-1}\\
&\equiv\sum_\mu [1+\mu]^{r-1}\\
&=p-1,
\end{align*}as required.
\end{proof}
\begin{corollary}\label{cor: sums of powers of lambda-mu}If $p>2$ and $r>1$ with $r\equiv1\pmod{p-1}$, and
if $t=v(r-1)$, then for all $\lambda\in\Fp$ we have
\begin{enumerate}
\item $\sum_{\mu\in\Fp}([\mu]-[\lambda])^r\equiv -[\lambda]
  rp\pmod{p^{t+2}}$, and
\item $\sum_{\mu\in\Fp}([\mu]-[\lambda])^{r-1}\equiv p-1\pmod{p^{t+1}}$.
\end{enumerate}
\end{corollary}
\begin{proof} If $\lambda=0$ then both statements are obvious.  If
  $\lambda\not=0$ then we simply take out a factor of $(-[\lambda])^r$ (resp.\
  $(-[\lambda])^{r-1}$) and observe that as $[\mu]$ runs over the
  Teichmueller lifts, so does $-[\mu]/[\lambda]$. This reduces both
  claims to the case $\lambda=-1$, which is Lemma \ref{lem: sums of powers
    of 1+lambda}.
\end{proof}

\begin{corollary}\label{cor: sum of powers of lx-mx+py}If $p>2$ and $r>1$ with $r\equiv1\pmod{p-1}$, and
if $t=v(r-1)$, then for all $\lambda\in\Fp$ we have 
$$\sum_{\mu\in\Fp}([\mu] x-[\lambda] x+py)^r\equiv -[\lambda] rpx^r+rp(p-1)x^{r-1}y\pmod{p^{t+2}\Zpbar[x,y]}.$$
\end{corollary}
\begin{proof}
Again by Lemma~\ref{binomial}(1), in order to compute  modulo $p^{t+2}$
 we only need to expand out the first
two terms of $(([\mu] x-[\lambda] x)+py)^r$, giving that the sum is
congruent to
$$\sum_{\mu\in\Fp}\left(([\mu] x-[\lambda] x)^r+rp([\mu] x-[\lambda]
x)^{r-1}y\right).$$ The result then follows from Corollary \ref{cor: sums of
  powers of lambda-mu}.
\end{proof}

\section{$p$-adic local Langlands: definitions and lemmas}\label{sec:
  p-adic LL}In this section we recall some of the basic definitions
and properties of the $p$-adic local Langlands correspondence. For
more details the reader could consult section~2 of \cite{MR2511912} or
any of the references therein. 
 
Say $r\in\Z_{\geq0}$. Let $K$ be the group $\GL_2(\Z_p)$, and for
$R$ a $\Z_p$-algebra let $\Symm^r(R^2)$ denote the space
$\oplus_{i=0}^rRx^{r-i}y^i$ of homogeneous polynomials in two variables
$x$ and $y$, with the action of $K$ given by
$$\begin{pmatrix}a&b\\c&d\end{pmatrix}x^{r-i}y^i=(ax+cy)^{r-i}(bx+dy)^i,$$
so $(\kappa v)(x,y)=v((x,y)\kappa)$.
Set $G=\GL_2(\Qp)$, and let $Z$ be its centre. If $V$ is an $R$-module
with an action of $K$, then extend
the action of $K$ to the group $KZ$ by letting $\smallmat{p}{0}{0}{p}$
act trivially, and let $I(V)$ denote the representation
$\ind_{KZ}^G(V)$ (compact induction).
Explicitly, $I(V)$ is the space of functions
$f:G\to V$ which have compact support modulo $Z$ and which
 satisfy $f(\kappa g)=\kappa(f(g))$ for all $\kappa\in KZ$. This space has
a natural action of $G$, defined by $(gf)(\gamma):=f(\gamma g)$.
Note that \S2.2 of~\cite{barthel-livne} explains that to give an $R$-linear
$G$-endomorphism of $I(V)$ is to give a certain compactly-supported
function $\phi:G\to\End_R(V)$ such that
$\phi(\kappa g\kappa')=\kappa\circ g\circ\kappa'$ for $g\in G$,
$\kappa\in KZ$, $\kappa'\in KZ$ (by Frobenius reciprocity). 

If $V=\Symm^r(R^2)$ for some integer $r\geq0$ and $\Z_p$-algebra $R$,
then there is a certain endomorphism $T$ of $I(V)$ which corresponds to the function
$G\to\End_R(V)$ which is supported on $KZ\smallmat{p}{0}{0}{1}KZ$ and
sends $\smallmat{p}{0}{0}{1}$ to the endomorphism of $\Symm^r(R^2)$
sending $F(x,y)$ to $F(px,y)$. Slightly more generally, if $V$
is the representation $\det^s\otimes\Symm^r(R^2)$ of~$K$ then
we can extend~$V$ to a representation of~$KZ$, note
that $I(V)=(\omega\circ\det^s)\otimes I(\Symm^r(R^2))$
and we define $T$ on $I(V)$ via its action on $I(\Symm^r(R^2))$.
Here $\omega:\Q_p^\times\to\Z_p^\times$ is the identity on $\Z_p^\times$
and sends~$p$ to~1.

We now establish some notation, following \cite{breuil2}.  Recall that
for $V$ a $\Z_p[K]$-module, the space $I(V)$ was defined previously to
be a certain space of functions $G\to V$.  We let $[g,v]$ denote the
(unique) element of $I(V)$ which is supported on $KZg^{-1}$, and which
satisfies $[g,v](g^{-1})=v$. One can check that $[g,v]$ corresponds
to $g\otimes v$ if we identify $I(V)$ with $R[G]\otimes_{R[KZ]}V$.
Note that $g[h,v]=[gh,v]$ for $g,h\in G$,
that $[g\kappa,v]=[g,\kappa v]$ for $\kappa\in KZ$, and that the $[g,v]$ span $I(V)$ as
an abelian group, as $g$ and $v$ vary.  

Now let $V=\Symm^r(R^2)$ for some
$\Z_p$-algebra $R$.
An easy consequence of
the definition of $T$ (cf. section 2 of \cite{breuil2}) is
that
\begin{equation*}\label{eqn:T following Breuil}
\tag{\ding{37}}
T[g,v]=\sum_{\lambda\in\Fp}\left[g\smallmat{p}{[\lambda]}{0}{1},v(x,-[\lambda]
x+py)\right]+\left[g\smallmat{1}{0}{0}{p},v(px,y)\right].  
\end{equation*}

Again we assume that $r\geq p$ and $r\equiv1$~mod~$p-1$.
By Lemma 3.2 of \cite{ashstevens}, there is a
$\GL_2(\Fp)$-equivariant surjection
$\Psi:\Symm^r\Fpbar^2\to\det\otimes\Symm^{p-2}\Fpbar^2$, such
that (using $X,Y$ for variables in $\Symm^{p-2}$)
 $$\Psi(f)=\sum_{s,t\in\Fp}f(s,t)(tX-sY)^{p-2}.$$ 

We now move on to the $p$-adic part of the story.
Say $k\in\Z_{\geq2}$ and $a\in\Zpbar$ with $v(a)>0$.

\begin{defn}Let $\Pi_{k,a}:=\ind_{KZ}^G\Symm^{k-2}(\Qpbar^2)/(T-a)$
(compact
induction, as before), and let  $\Theta_{k,a}$ be the image
of $\ind_{KZ}^G\Symm^{k-2}(\Zpbar^2)$ in
$\Pi_{k,a}$.

\end{defn}

If $a\not=\pm p^{k/2}(1+p^{-1})$ then we claim that $\Pi_{k,a}$ is irreducible
 and $\Theta_{k,a}$ is a lattice in it.
Indeed, irreducibility of $\Pi_{k,a}$ is proved in Proposition~3.2.1(i)
of~\cite{breuil2}, the existence of a $G$-stable lattice is proved
in Corollaire~5.3.4 of~\cite{berger-breuil}, and now the fact
that $\Theta_{k,a}$ is a lattice follows from the fact that it
is finitely-generated as a $\Zpbar[G]$-module (hence contained in
a lattice) and visibly spans $\Pi_{k,a}$.
Because of Theorem~3.2.1 of~\cite{MR2642408} (which deals
with $k=3$ and $k=p+2$), and Theorem~1.6 of~\cite{MR2511912}
and the comments following it (which deal with $k\not\equiv3$~mod~$p-1$),
we are only really concerned in this paper in the case $k\geq 2p+1$,
$k\equiv3$~mod~$p-1$ and $0<v(a)<1$, which implies $a\not=\pm(1+p^{-1})p^{k/2}$
anyway. So let us assume $k\geq 2p+1$ and $k\equiv3$~mod~$p-1$.
To simplify notation set $r=k-2$, so $r\equiv1$~mod~$p-1$.
Now by Corollary 5.1 of \cite{MR2511912},
the natural surjection
$I(\Symm^r\Fpbar^2)\to\overline{\Theta}_{k,a}$ factors through the map
$I(\Symm^r\Fpbar^2)\to I(\det\otimes\Symm^{p-2}\Fpbar^2)$ induced by
$\Psi$. The key input we need from the $p$-adic local Langlands
correspondence is the following lemma.
\begin{lem}
  \label{lem: p-adic LL lets us read off rhobar}
Assume $k\geq 2p+1$, $k\equiv3$~mod~$p-1$, and $0<v(a)<1$.
  \begin{enumerate}
  \item If $\overline{\Theta}_{k,a}$ is a quotient of
$I(\det\otimes\Symm^{p-2}\Fpbar^2)/T$, then
$\Vbar_{k,a}\cong\ind(\omega_2^2)$ is irreducible.
\item If $\overline{\Theta}_{k,a}$ is a quotient of
$I(\det\otimes\Symm^{p-2}\Fpbar^2)/(T^2-cT+1)$ for some $c\in\Fpbar$, then
$\Vbar_{k,a}$ is reducible, and $\omega^{-1}\otimes\Vbar_{k,a}$ is an
unramified reducible representation, and the trace of (both arithmetic
and geometric) $\Frob_p$ is
$c$.\end{enumerate}

\end{lem}
\begin{proof}
  This may be proved in exactly the same way as Proposition 3.3 of
  \cite{MR2511912}. Note that both arithmetic and geometric Frobenius
have the same trace in case~(2), because $\omega^{-1}\otimes\Vbar_{k,a}$
has trivial determinant (as $k\equiv3$~mod~$p-1$).
\end{proof}

\section{Computations with Hecke operators}\label{sec: Hecke}

We again assume throughout this section that $p>2$ is an odd prime
and $r>p$ is an integer such that $r\equiv 1\pmod{p-1}$. 
We start with a couple of results about the map
$\Psi:\Symm^r(\Fpbar^2)\to\det\otimes\Symm^{p-2}(\Fpbar^2)$
defined in the previous section.

\begin{lemma}\label{lem:computation of Psi on elements}
  \begin{enumerate}
  \item  $\Psi(y^r)=0$.
  \item $\Psi(x^r)=0$.
  \item $\Psi(x^{r-1}y)=X^{p-2}$.
 \end{enumerate}
\end{lemma}
\begin{proof}
  $\Psi(y^r)=\sum_{s,t\in\Fp}t^r(tX-sY)^{p-2}=\sum_{s,t\in\Fp}t(tX-sY)^{p-2}$.
  Expanding out  using the binomial theorem and using the
  fact that $\sum_ss^n=0$ for $n=0,1,2,\ldots,p-2$, this sum is
  zero. Since $\Psi$ is $\GL_2(\Fp)$-equivariant, $\Psi(x^r)$ is also zero.

Since $r\equiv 1\pmod{p-1}$, we have
\begin{align*}\Psi(x^{r-1}y)&=\sum_{s,t\in\Fp}s^{p-1}t(tX-sY)^{p-2}\\
&=\sum_{s\not=0,t\in\Fp}t(tX-sY)^{p-2}
\end{align*}
and $\sum_{s\not=0}s^n$ is not zero if $n=0$ (although it is for
$1\leq n\leq p-2$ as before) so expanding out we get $-\sum_{t\in\Fp}t(tX)^{p-2}$
which is $-\sum_{t\not=0}X^{p-2}=X^{p-2}$.

\end{proof}

\begin{lem}\label{lem:computation of T in Sym p-2} In $I(\Symm^{p-2}(\Fpbar^2))$ we have
  \begin{enumerate}
  \item
    $T[1,X^{p-2}]=\sum_{\mu\in\Fp}[\smallmat{p}{[\mu]}{0}{1},X^{p-2}]$, and
  \item $T^2[1,X^{p-2}]=\sum_{\lambda,\mu\in\Fp}[\smallmat{p^2}{p[\mu]+[\lambda]}{0}{1},X^{p-2}]$.
  \end{enumerate}
\end{lem}
\begin{proof} This is immediate from ~(\ref{eqn:T following Breuil}).
\end{proof}

\begin{lemma}\label{lem: first T-a computation.}Assume that $p>2$ and that $r>p$ with
  $r\equiv 1\pmod{p-1}$. Set $t=v(r-1)$ and suppose
$a\in\Qpbar$ with $v(a)>0$. Then
\begin{align*}(T-a)[g,y^r-x^{r-1}y]&\equiv[g\smallmat{p}{0}{0}{1},-px^{r-1}y]\\&+\sum_{\lambda\not=0}[g\smallmat{p}{\lambda}{0}{1},(r-1)px^{r-1}y]+[g\smallmat{1}{0}{0}{p},y^r]-[g,a(y^r-x^{r-1}y)]\end{align*} modulo
$p^{t+2}I(\Symm^r(\Zpbar^2))$.
\end{lemma}
\begin{proof} This is immediate from (\ref{eqn:T following Breuil})
and Lemma~\ref{first congruence mod p^t+2}.

\end{proof}
We have $t=v(r-1)$; say $a\in\Qpbar$ satisfies $0<v(a)<1$,
and set $t_0=\min\{t+1+v(a),v(a^2-rp)\}$.
\begin{lemma}\label{Xg}Assume that $p>2$ and that $r>p$ with
  $r\equiv 1\pmod{p-1}$.  If
  $\varphi_g=\sum_{j=0}^{N}[g\smallmat{p^j}{0}{0}{1},a^j(y^r-x^{r-1}y)]$
  where $N>t_0/v(a)$, then in $I(\Symm^r(\Zpbar^2))$ we have
$$(T-a)\varphi_g\equiv\sum_{\lambda\in\Fp}[g\smallmat{p}{[\lambda]}{0}{1},(r-1)px^{r-1}y]
+[g\smallmat{1}{0}{0}{p},y^r]+[g,ax^{r-1}y]\pmod{p^{t_0}}.$$
\end{lemma}
\begin{proof} Throughout the proof we will write $\sum_{j\ge 0}$ rather than
  keeping track of the upper index of our sums, as the implied terms
  will all be zero modulo $p^{t_0}$. Since $t+2>t+1+v(a)\ge t_0$, we can apply
  Lemma \ref{lem: first T-a computation.}. Noting that if $j\ge 1$,
  $v(a^j(r-1)p)\ge t+1+v(a)\ge t_0$, we see that modulo $p^{t_0}$,
  $(T-a)\varphi_g$ is just
\begin{align*}
&[g\smallmat{p}{0}{0}{1},-px^{r-1}y]+\sum_{\lambda\not=0}[g\smallmat{p}{[\lambda]}
{0}{1},(r-1)px^{r-1}y]+[g\smallmat{1}{0}{0}{p},y^r]-[g,a(y^r-x^{r-1}y)]\\
&+\sum_{j\ge 1}[g\smallmat{p^{j+1}}{0}{0}{1},-pa^jx^{r-1}y]\\
&+\sum_{j\ge 1}[g\smallmat{p^{j-1}}{0}{0}{1},a^jy^r]\\
&+\sum_{j\ge 1}[g\smallmat{p^j}{0}{0}{1},-a^{j+1}y^r+a^{j+1}x^{r-1}y]
\end{align*}
which rearranges to
\begin{align*}
&\sum_{\lambda\not=0}[g\smallmat{p}{[\lambda]}{0}{1},(r-1)px^{r-1}y]\\
&+[g\smallmat{1}{0}{0}{p},y^r]+[g,ax^{r-1}y]\\
&+\sum_{s\ge 1}[g\smallmat{p^s}{0}{0}{1},-pa^{s-1}x^{r-1}y]\\
&+\sum_{s\ge 1}[g\smallmat{p^s}{0}{0}{1},a^{s+1}y^r]\\
&+\sum_{s\ge1}[g\smallmat{p^s}{0}{0}{1},-a^{s+1}y^r+a^{s+1}x^{r-1}y]
\end{align*}
where we have changed variables from $j$ to $s$ to make
all the sums involve $g\smallmat{p^s}{0}{0}{1}$, and put two
of the ``initial'' terms into the sums.  Pressing on, we get two
terms in the sums cancelling and we are left with
\begin{align*}
&\sum_{\lambda\not=0}[g\smallmat{p}{[\lambda]}{0}{1},(r-1)px^{r-1}y]\\
&+[g\smallmat{1}{0}{0}{p},y^r]+[g,ax^{r-1}y]\\
&+\sum_{s\ge 1}[g\smallmat{p^s}{0}{0}{1},a^{s-1}(a^2-p)x^{r-1}y]\\
\end{align*}
By the definition of $t_0$ we have $ap(r-1)\equiv 0\pmod{p^{t_0}}$ and
$a^2-rp\equiv 0\pmod{p^{t_0}}$, so we see that if $s\ge 2$ we have
  $a^{s-1}(a^2-p)\equiv a^{s-2}(a(a^2-rp)+ap(r-1))\equiv 0 \pmod{p^{t_0}}$.
Thus we can simplify further to
\begin{align*}
&\sum_{\lambda\not=0}[g\smallmat{p}{[\lambda]}{0}{1},(r-1)px^{r-1}y]\\
&+[g\smallmat{1}{0}{0}{p},y^r]+[g,ax^{r-1}y]\\
&+[g\smallmat{p}{0}{0}{1},(a^2-p)x^{r-1}y].
\end{align*}
Finally, since $a^2\equiv rp\pmod{p^{t_0}}$, the last term can be
inserted into the sum by allowing
$\lambda=0$.\end{proof}

\begin{lemma}
  \label{cor: (T-a)X}Assume that $p>2$ and that $r>p$ with
  $r\equiv 1\pmod{p-1}$.  If
  \[\varphi=-p\varphi_1+\sum_{\mu\in\Fp} a\varphi_{\smallmat{p}{[\mu]}{0}{1}}+[1,\sum_{\mu\in\Fp} ([\mu] x+y)^r-rpx^{r-1}y]\]
  (where $\varphi_g$ is as in the statement of Lemma~\ref{Xg}), and
if $t_1=t_0+\min\{v(a),1-v(a)\}>t_0$, then 
\begin{align*}(T-a)\varphi&\equiv
\sum_{\lambda,\mu\in\Fp}[\smallmat{p^2}{p[\lambda]+[\mu]}{0}{1},a(r-1)px^{r-1}y]\\
&+[1,ap(r-1)x^{r-1}y]\\
&+\sum_{\lambda\in\Fp}[\smallmat{p}{[\lambda]}{0}{1},(a^2-rp)x^{r-1}y]\pmod{p^{t_1}}.
\end{align*}
\end{lemma}
\begin{proof} First we note that $t+2\ge t_1$
 (because $t+2=t+1+v(a)+(1-v(a))\geq t_0+(1-v(a))\geq t_1$).
By Lemma \ref{lem: r at least t+3} we have $r\ge t+3$
and hence $r\ge t_1+1>t_1$, so $p^r\equiv 0\pmod{p^{t_1}}$. We also
see from Lemma \ref{binomial} and the inequality $r-1\ge t_1$
that \[\sum_{\mu\in\Fp}([\mu]px+y)^r\equiv py^r\pmod{p^{t_1}}.\]
 Using these facts and 
Lemma~\ref{Xg} (for the first two terms in the definition of
  $\varphi$) and (\ref{eqn:T following Breuil}) (for the final term), we see
  that modulo $p^{t_1}$, we have
  \begin{align*}
(T-a)\varphi&\equiv
\left[\smallmat{1}{0}{0}{p},-py^r\right]+\left[1,-pax^{r-1}y\right]\\
&+[1,a\sum_{\mu\in\Fp} ([\mu] x+y)^r]+\sum_{\mu\in\Fp}\left[\smallmat{p}{[\mu]}{0}{1},a^2x^{r-1}y\right]\\
&+\sum_{\lambda,\mu\in\Fp}\left[\smallmat{p^2}{p[\lambda]+[\mu]}{0}{1},a(r-1)px^{r-1}y\right]\\
&+[1,-a\sum_{\lambda\in\Fp}([\lambda] x+y)^r+arpx^{r-1}y]\\
&+\sum_{\lambda\in\Fp}\left[\smallmat{p}{[\lambda]}{0}{1},\left(\sum_{\mu\in\Fp}([\mu] x-[\lambda] x+py)^r\right)-rpx^{r-1}(-[\lambda] x+py)\right]
+[\smallmat{1}{0}{0}{p},py^r].
\end{align*}
Now some terms cancel, and we get
\begin{align*}
&\sum_{\lambda,\mu\in\Fp}[\smallmat{p^2}{p[\lambda]+[\mu]}{0}{1},a(r-1)px^{r-1}y]\\
&+[1,ap(r-1)x^{r-1}y]\\
&+\sum_{\lambda\in\Fp}\left[\smallmat{p}{[\lambda]}{0}{1},\left(\sum_{\mu\in\Fp}([\mu] x-[\lambda] x+py)^r\right)+a^2x^{r-1}y-rpx^{r-1}(-[\lambda] x+py)\right].
\end{align*}
Again noting that $t+2\ge t_1$, by Corollary \ref{cor: sum of
  powers of lx-mx+py} we have\[\sum_{\mu\in\Fp}([\mu] x-[\lambda] x+py)^r\equiv-[\lambda]
rpx^r+rp(p-1)x^{r-1}y\pmod{p^{t_1}},\] and the result follows.
\end{proof}

\begin{cor}\label{cor: description of thetabar in the three cases}  Assume that $p>2$ and that $r>p$ with
  $r\equiv 1\pmod{p-1}$, and that $0<v(a)<1$.
  \begin{enumerate}
  \item If $v(r-1)+1+v(a)>v(a^2-rp)$, then $\Thetabar_{k,a}$ is a
    quotient of $I(\det\otimes\Symm^{p-2}\Fpbar^2)/T$.
  \item If $v(r-1)+1+v(a)\le v(a^2-rp)$, then $\Thetabar_{k,a}$ is a
    quotient of $I(\det\otimes\Symm^{p-2}\Fpbar^2)/(T^2-\overline{\tau}T+1)$  where $\tau=\frac{rp-a^2}{ap(r-1)}$.
 \end{enumerate}

\end{cor}
\begin{proof}
  Set $\psi=p^{-t_0}\varphi$, with $\varphi$ as in Lemma~\ref{cor: (T-a)X}. By
  the definition of $t_0$, we see that both $v(a(r-1)p)$, and
  $v(a^2-rp)$ are at least $t_0$, so that $(T-a)\psi$ is integral, by
  Lemma~\ref{cor: (T-a)X}. Thus $\overline{(T-a)\psi}$ is in the
  kernel of the natural map
  $I(\Symm^r\Fpbar^2)\onto\Thetabar_{k,a}$. We will now compute
  $\Psi(\overline{(T-a)\psi})$ in both cases, and hence
  deduce the claim.
  \begin{enumerate}
  \item If $v(r-1)+1+v(a)>v(a^2-rp)$, then we see from
Lemma~\ref{cor: (T-a)X} that  $\overline{(T-a)\psi}$ is a unit times
    $\sum_{\lambda\in\Fp}[\smallmat{p}{[\lambda]}{0}{1},x^{r-1}y]$. By
    Lemma \ref{lem:computation of Psi on elements}(3) and Lemma
    \ref{lem:computation of T in Sym p-2}(1), we see that
    $\Psi(\overline{(T-a)\psi})$ is a unit times $T[1,X^{p-2}]$ and the
    result follows. (Note that $\det\otimes\Symm^{p-2}\Fpbar^2$ is
    irreducible, and in particular is generated by $X^{p-2}$.)
  
  \item If $v(r-1)+1+v(a)\le v(a^2-rp)$, then writing
    $\tau=\frac{rp-a^2}{ap(r-1)}$, we see that  $\overline{(T-a)\psi}$ is a
    unit times \begin{align*}
&\sum_{\lambda,\mu\in\Fp}[\smallmat{p^2}{p[\lambda]+[\mu]}{0}{1},x^{r-1}y]\\
&+[1,x^{r-1}y]\\
&-\overline{\tau}\sum_{\lambda\in\Fp}[\smallmat{p}{[\lambda]}{0}{1},x^{r-1}y].
\end{align*}By  Lemma \ref{lem:computation of Psi on elements}(3) and Lemma
    \ref{lem:computation of T in Sym p-2}, we see that
    $\Psi(\overline{(T-a)\psi})$ is a unit times $(T^2-\overline{\tau}T+1)[1,X^{p-2}]$,
    as required.

  \end{enumerate}

\end{proof}

\begin{cor}
  \label{cor: explicit description of Vbar in the 3 cases}Assume that $p>2$ and that $r>p$ with
  $r\equiv 1\pmod{p-1}$, and that $0<v(a)<1$.
  \begin{enumerate}
  \item If $v(r-1)+1+v(a)>v(a^2-rp)$, then $\Vbar_{k,a}\cong\ind(\omega_2^2)$ is irreducible.
  \item If $v(r-1)+1+v(a)\le v(a^2-rp)$, then
    $\omega^{-1}\otimes\Vbar_{k,a}$ is unramified, and the trace of
    $\Frob_p$ on this representation is  $\overline{t}$, where $t=\frac{rp-a^2}{ap(r-1)}$.
  \end{enumerate}
\end{cor}
\begin{proof}This is immediate from Lemma \ref{lem: p-adic LL lets us
    read off rhobar} and Corollary \ref{cor: description of thetabar in the three cases}.  
\end{proof}

\bibliographystyle{amsalpha}
\bibliography{buzzardgee}

\def\cftil#1{\ifmmode\setbox7\hbox{$\accent"5E#1$}\else
  \setbox7\hbox{\accent"5E#1}\penalty 10000\relax\fi\raise 1\ht7
  \hbox{\lower1.15ex\hbox to 1\wd7{\hss\accent"7E\hss}}\penalty 10000
  \hskip-1\wd7\penalty 10000\box7}
\providecommand{\bysame}{\leavevmode\hbox to3em{\hrulefill}\thinspace}
\providecommand{\MR}{\relax\ifhmode\unskip\space\fi MR }
\providecommand{\MRhref}[2]{%
  \href{http://www.ams.org/mathscinet-getitem?mr=#1}{#2}
}
\providecommand{\href}[2]{#2}
\begin{thebibliography}{Bre03}

\bibitem[AS86]{ashstevens}
Avner Ash and Glenn Stevens, \emph{Modular forms in characteristic {$l$} and
  special values of their {$L$}-functions}, Duke Math. J. \textbf{53} (1986),
  no.~3, 849--868.

\bibitem[Ber08]{ber05}
Laurent Berger, \emph{Repr\'{e}sentations modulaires de
  $\operatorname{GL}_2(\mathbb{Q}_p)$ et repr\'{e}sentations galoisiennes de
  dimension 2}, Asterisque (to appear) (2008).

\bibitem[Ber10]{MR2642408}
\bysame, \emph{Repr\'esentations modulaires de {${\rm GL}_2(\bold Q_p)$} et
  repr\'esentations galoisiennes de dimension 2}, Ast\'erisque (2010), no.~330,
  263--279.

\bibitem[Ber11]{bergerSemBourb}
\bysame, \emph{La correspondance de {L}anglands locale $p$-adique pour
  $\operatorname{GL}_2(\mathbb{Q}_p)$}, Asterisque \textbf{339} (2011).

\bibitem[Ber12]{BergerLocCst}
\bysame, \emph{Local constancy for the reduction mod $p$ of 2-dimensional
  crystalline representations}, Bulletin of the London Mathematical Society (to
  appear) (2012).

\bibitem[BG09]{MR2511912}
Kevin Buzzard and Toby Gee, \emph{Explicit reduction modulo {$p$} of certain
  two-dimensional crystalline representations}, Int. Math. Res. Not. IMRN
  (2009), no.~12, 2303--2317.

\bibitem[BL94]{barthel-livne}
L.~Barthel and R.~Livn{\'e}, \emph{Irreducible modular representations of
  {${\rm GL}\sb 2$} of a local field}, Duke Math. J. \textbf{75} (1994), no.~2,
  261--292.

\bibitem[Bre03]{breuil2}
Christophe Breuil, \emph{Sur quelques repr\'esentations modulaires et
  {$p$}-adiques de {${\rm GL}\sb 2(\bold Q\sb p)$}. {II}}, J. Inst. Math.
  Jussieu \textbf{2} (2003), no.~1, 23--58.

\end{thebibliography}

\end{document}